\documentclass[12pt]{article}
\usepackage{amssymb}
\usepackage{amsfonts}
\usepackage{amsmath}
\usepackage{geometry}

\newtheorem{theorem}{Theorem}

\newtheorem{lemma}[theorem]{Lemma}

\newenvironment{proof}[1][Proof]{\noindent\textbf{#1.} }{\ \rule{0.5em}{0.5em}}

\begin{document}

\title{Generalized Lyapunov-Sylvester operators for Kuramoto-Sivashinsky
Equation}
\author{Abdelhamid BEZIA \and Anouar Ben Mabrouk}
\maketitle

\begin{abstract}
A numerical method is developed leading to algebraic systems based on
generalized Lyapunov-Sylvester operators to approximate the solution of
two-dimensional Kuramoto-Sivashinsky equation. It consists of an order
reduction method and a finite difference discretization which is proved to
be uniquely solvable, stable and convergent by using Lyapunov criterion and
manipulating generalized Lyapunov-Sylvester operators. Some numerical
implementations are provided at the end to validate the theoretical results.%
\newline


\end{abstract}

\section{Introduction}

The present paper is devoted to the development of a computational method
based on two-dimensional finite difference scheme to approximate the
solution of the nonlinear Kuramoto-Sivashinsky equation%
\begin{equation}
\frac{\partial u}{\partial t}=q\Delta u-\kappa \Delta ^{2}u+\lambda
\left\vert \nabla u\right\vert ^{2},\ ((x,y),t)\in \Omega \times
(t_{0},+\infty ),  \label{E-Kura-Shiva}
\end{equation}%
with initial conditions
\begin{equation}
u(x,y,t_{0})=\varphi (x,y)\;;\quad (x,y)\in \Omega  \label{eqn1-3}
\end{equation}%
and boundary conditions
\begin{equation}
\frac{\partial u}{\partial \eta }\left( x,y,t\right) =0\;;\quad ((x,y),t)\in
\partial \Omega \times (t_{0},+\infty ),  \label{eqn1-4}
\end{equation}%
on a rectangular domain $\Omega =[L_{0},L_{1}]\times \lbrack L_{0},L_{1}]$
in $\mathbb{R}^{2},$ $t_{0}\geq 0$ is a real parameter fixed as the initial
time. $\frac{\partial }{\partial t}$ is the time derivative, $\nabla $ is
the space gradient operator and $\Delta =\frac{\partial ^{2}}{\partial x^{2}}%
+\frac{\partial ^{2}}{\partial y^{2}}$ is the Laplace operator in $\mathbb{R}%
^{2}$, $q$,$\,\kappa $,$\,\lambda $ are real parameters. $\varphi $ and $%
\psi $ are twice differentiable real valued functions on $\overline{\Omega }$%
.

We propose to apply an order reduction of the derivation and thus to solve a
coupled system of equation involving second order differential operators. We
set $v=qu-\kappa \Delta u$ and thus we have to solve the system%
\begin{equation}
\left\{
\begin{array}{l}
\frac{\partial u}{\partial t}=\Delta v+\lambda \left\vert \nabla
u\right\vert ^{2},\quad (x,y,t)\in \Omega \times (t_{0},+\infty ) \\
v=qu-\kappa \Delta u,\quad (x,y,t)\in \Omega \times (t_{0},+\infty ) \\
\left( u,v\right) (x,y,t_{0})=\left( \varphi ,\psi \right) (x,y),\quad
(x,y)\in \overline{\Omega } \\
\overrightarrow{\nabla }(u,v)(x,y,t)=0,\quad (x,y,t)\in \partial \Omega
\times (t_{0},+\infty )%
\end{array}%
\right.  \label{Problem1}
\end{equation}

The Kuramoto-Sivashinsky equation (KS) is one of the most famous equations
in math-physics for many decades. It has its origin in the work of Kuramoto
since the 70-th decade of the 20-th century in his study of
reaction-diffusion equation. The equation was then considered by Sivashinsky
in modeling small thermal diffusion instabilities for laminar flames and
modeling the reference flux of a film layer on an inclined plane. Since then
the KS equation has experienced a growing development in theoretical
mathematics, numerical as well as physical mechanics, nonlinear physics,
hydrodynamics, chemistry, plasma physics, particle distributions advection,
surface morphology, ...etc. See \cite{Benachour}, \cite{Hansen}, \cite{Hong}%
, \cite{Jayaprakash}, \cite{Kuramoto}, \cite{Nadjafikhah}, \cite{Procaccia},
\cite{Rost}, \cite{Sivashinsky}, \cite{Sivashinsky2}. In \cite{Nadjafikhah},
the symmetry problem of the model was studied based on the theory of Lie
algebras. In \cite{Benachour}, an anisotropic version of the KS equation has
been proposed leading to global resolutions of the equation on rectangular
domaines. Sufficient conditions were given for the existence of global
solution.

From the dimensional point of view, KS equation has been widely studied in
one dimension, \cite{Giacomelli}, \cite{Goodman}, \cite{Ilyashenko}, \cite%
{Nicolaenko}, \cite{Otto}, \cite{Temam}. However, this equation since its
appearance is related to the modeling of flame spread which is a
two-dimensional problem. In this context, an important result was developed
in \cite{Sell} where the authors showed by adapting the method developed in
\cite{Raugel} the existence of a set of bounded solutions on a rectangular
domain. This major importance of the two-dimensional model was a main
motivation behind this work. A model representing a nonlinear dynamical
system defined in a two-dimensional space is considered, where the solution $%
u(x,y,t)$ satisfies a fourth order partial differential equation of the form
(\ref{E-Kura-Shiva}), where $u$ is the height of the interface, $q$ is a
pre-factor proportional to the coefficient of surface tension expressed by $%
q\nabla ^{2}u$. The quantity $\kappa \nabla ^{4}u$ represents the result of
the diffusion surface due to the chemical potential gradient induced
curvature. The pre-factor $\kappa $ represents the surface diffusion term.
The quantity $\lambda |\nabla u|^{2}$ is due to the existence of overhangs
and vacancies during the deposition process. Finally, the quantity $\nu
\nabla ^{2}u+\lambda \nabla u|^{2}u$ is referred to as modeling the effect
of deposited atoms. cf. \cite{Hong}.

In the present work we propose to serve of algebraic operators to
approximate the solutions of the Kuramoto-Sivashinsky(K-S) equation in the
two patial and one temporal dimesion without adapting classical developments
based on separation of variables, radial solutions, tridiagonal operators,
... etc.

This was onefold of the present paper. A second crucial idea is to transform
the continuous K-S equation into a generalized Lyapunov-Sylvester equation
of the form
\begin{equation}
\sum_{i}A_{i}X_{n}B_{i}=C_{n}  \label{Lyp-Sylv-Equa}
\end{equation}%
where $A_{i}$ and $B_{i}$ are appropriate matrices depending on the
discretization procedure and the problem parameters. $X_{n}$ represents the
numerical solution at time $n$ and $C_{n}$ is usually depending on the past
values $X_{k},$ $k\leq n-1$ of the the solution. The equation $\left( \ref%
{Lyp-Sylv-Equa}\right) $ is known as generalized Lyapunov-Sylvester
equation. Such equations have their origin in the work of Sylvester on
classical matrices equations. In the particular case

\begin{equation}
\sum A_{i}X\overline{A}_{i}^{T}=C
\end{equation}%
$\bigskip $the equation is knwon restrictively as Lyapunov one \cite%
{Simoncini}. Generally speaking, the equation

\begin{equation}
\sum_{i}A_{i}XB_{i}=C  \label{LSG1}
\end{equation}%
\bigskip\ is very difficult to be inverted and remains an open problem in
algebra. Nevertheless, some works have been developed and proved that under
suitable conditions on the coefficient matrices, one may get a unique
solution, but it's exact computation remains hard. It necessitates to
compute eigenvalues and precisely bounds/estimates of eigenvalues or direct
inverses of big matrices which remains a complicated problem and usually
inappropriate.

In \cite{Lancaster}, a native method to solve $\left( \ref{LSG1}\right) $ is
investigated based on Kronecker product and equivalent matrix-vector
equation. The Sylvester's equation is transformed into a linear equivalent
one on the form $Gx=c,$ with a matrix $G$ obtained by tensor products issued
from the $A_{j}^{\prime }s$ and the $B_{j}^{\prime }s$. However, the general
case remained already complicated. The authors have been thus restricted to
special cases where the matrices $A_{j}$ and $B_{j}$ are scalar polynomials
based on spacial and fixed matrices $A\ $and $B.$ Denote by $\sigma \left(
A\right) $ the spectrum of $A$ and $\sigma (B)$ the one of $B$, the spectrum
$\sigma (G)$ may then be determined in terms these spectra. Indeed, with the
assumptions the $A_{j}^{\prime }s$ and the $B_{j}^{\prime }s,$ the equation $%
\left( \ref{LSG1}\right) $ can be writen as

\begin{equation*}
\sum_{j,k}\alpha _{j,k}A^{j}XB^{k}=C,
\end{equation*}%
where $\alpha _{jk}$ are complexe numbers. Hence, the tensor matrix $G$ will
be written on the form $G=\phi \left( A,B\right) ,$ where $\phi $ is the
2-variables polynomial $\phi \left( x,y\right) =\sum_{j,k}\alpha
_{j,k}x^{i}y^{k}.$Thus, $G$ is singular if and only if $\phi \left( \lambda
,\mu \right) =0$ for some eigenvalues $\lambda $ and $\mu $ of $A$ and $B$
respectively.

In the case of square matrices an other criterion of existence of the
solution of $\left( \ref{LSG1}\right) $ was pointed out by Roth \cite{Roth}.
It was shown that the solution is unique if and only if the matrices%
\begin{equation*}
\left[
\begin{array}{cc}
A & C \\
0 & B%
\end{array}%
\right] ,\ \ \ \text{and \ \ }\left[
\begin{array}{cc}
A & 0 \\
0 & B%
\end{array}%
\right]
\end{equation*}%
are similar.

However, in numerical stydies of PDE's we may be confronted with matrices $G$
where the computation of spectral properties are not easy and necessitate
enormes calculus and sometimes, induces slow algorithms and bed convergence
rates. Thus, one motivation here and issued from \cite{Bezia} consists to
resolve such problem and prove the invertibility of the algebraic operator
yielded in the numerical scheme by applaying topological method Instead of
using classical ones such as tri-diagonal transformations. We thus aim to
prove that generilazed Lyapunov-Sylvester operators can be good candidates
for investigating numerical solutions of PDEs in multi-dimensional spaces.

The paper is organized as follows. In next section the discretization of $%
\left( \ref{Problem1}\right) $ is developed, the solvability of the scheme
is analyzed. In section 4, the consistency of the method is shown and next,
the stability and convergence are proved based on Lyapunov method and
Lax-Richtmyer theorem. Finally, a numerical implementation is provided
leading to the computation of the numerical solution and error estimates.

\section{The numerical scheme}

Let $J\in \mathbb{N}^{\ast }$ and $h=\frac{L_{1}-L_{0}}{J}$ be the space
step. Denote for $\left( j,m\right) \in I^{2}=\left\{ 0,1,...,J\right\}
^{2}, $ $x_{j}=L_{0}+jh$ and $y_{m}=L_{0}+mh$. Next, let $l=\Delta t$ be the
time step and $t_{n}=t_{0}+nl,$ $n\in \mathbb{N}$ be the time grid. We
denote also $\widetilde{\Omega }=\left\{ \left( x_{j},y_{m},t_{n}\right)
;\left( x_{j},y_{m},t_{n}\right) \in I^{2}\times \mathbb{N}\right\} $ the
associated discrete space. Finally, for $\left( j,m\right) \in I^{2}$ and $%
n\geq 0$, $u_{j,m}^{n}$ denotes the net function $u(x_{j},y_{m},t_{n})$ and $%
U_{j,m}^{n} $ is the numerical solution.

The following discrete approximations will be applied for the different
differential operators involved in the problem. For time derivatives, we set
\begin{equation*}
\frac{\partial u}{\partial t}\leadsto \frac{U_{j,m}^{n+1}-U_{j,m}^{n-1}}{2l}
\end{equation*}%
and for space derivatives, we shall use for the one order
\begin{equation*}
\frac{\partial u}{\partial x}\leadsto \frac{U_{j+1,m}^{n}-U_{j-1,m}^{n}}{2h}%
,\ \hbox{ \ \ and \ \ }\frac{\partial u}{\partial y}\leadsto \frac{%
U_{j,m+1}^{n}-U_{j,m-1}^{n}}{2h}
\end{equation*}%
and for the second order ones, we set%
\begin{equation*}
\frac{\partial ^{2}u}{\partial x^{2}}\leadsto \frac{U_{j+1,m}^{n,\alpha
}-2U_{j,m}^{n,\alpha }+U_{j-1,m}^{n,\alpha }}{h^{2}},\ \hbox{and\ \ \ }\frac{%
\partial ^{2}u}{\partial y^{2}}\leadsto \frac{U_{j,m+1}^{n,\alpha
}-2U_{j,m}^{n,\alpha }+U_{j,m-1}^{n,\alpha }}{h^{2}}.
\end{equation*}%
Finally, for $n\in \mathbb{N}^{\ast }$ and $\alpha \in \mathbb{R}$, we
denote
\begin{equation*}
U^{n,\alpha }=\alpha U^{n-1}+\left( 1-2\alpha \right) U^{n}+\alpha U^{n+1}
\end{equation*}%
to designate the estimation of $U_{j,m}^{n}$ with an $\alpha $%
-extrapolation/interpolation barycenter method. By applying these discrete
approximations, we obtain
\begin{eqnarray*}
U_{j,m}^{n+1}-U_{j,m}^{n-1} &=&\frac{2l}{h^{2}}\left[ V_{j-1,m}^{n,\beta
}-2V_{j,m}^{n,\beta }+V_{j+1,m}^{n,\beta }+V_{j,m-1}^{n,\beta
}-2V_{j,m}^{n,\beta }+V_{j,m+1}^{n,\beta }\right]  \\
&&+\frac{2l}{h^{2}}\lambda \left[ \frac{1}{4}\left(
U_{j+1,m}^{n}-U_{j-1,m}^{n}\right) ^{2}+\frac{1}{4}\left(
U_{j,m+1}^{n}-U_{j,m-1}^{n}\right) ^{2}\right] .
\end{eqnarray*}%
In what follows, we set
\begin{equation*}
F_{j,m}^{n}=\frac{1}{4}\left[ \left( U_{j+1,m}^{n}-U_{j-1,m}^{n}\right)
^{2}+\left( U_{j,m+1}^{n}-U_{j,m-1}^{n}\right) ^{2}\right]
\end{equation*}%
We thus get
\begin{eqnarray*}
U_{j,m}^{n+1}-U_{j,m}^{n-1} &=&\sigma \lbrack \beta V_{j-1,m}^{n+1}+\left(
1-2\beta \right) V_{j-1,m}^{n}+\beta V_{j-1,m}^{n-1} \\
&&-2\beta V_{j,m}^{n+1}-2\left( 1-2\beta \right) V_{j,m}^{n}-2\beta
V_{j,m}^{n-1} \\
&&+\beta V_{j+1,m}^{n+1}+\left( 1-2\beta \right) V_{j+1,m}^{n}+\beta
V_{j+1,m}^{n-1} \\
&&+\beta V_{j,m-1}^{n+1}+\left( 1-2\beta \right) V_{j,m-1}^{n}+\beta
V_{j,m-1}^{n-1} \\
&&-2\beta V_{j,m}^{n+1}-2\left( 1-2\beta \right) V_{j,m}^{n}-2\beta
V_{j,m}^{n-1} \\
&&+\beta V_{j,m+1}^{n+1}+\left( 1-2\beta \right) V_{j,m+1}^{n}+\beta
V_{j,m+1}^{n-1}]+\sigma \lambda F_{j,m}^{n}
\end{eqnarray*}%
where $\sigma =\frac{2l}{h^{2}}$. Otherwise, this can be written as
\begin{eqnarray*}
&&U_{j,m}^{n+1}-\sigma \beta \left[
V_{j-1,m}^{n+1}-2V_{j,m}^{n+1}+V_{j+1,m}^{n+1}+V_{j,m-1}^{n+1}-2V_{j,m}^{n+1}+V_{j,m+1}^{n+1}%
\right]  \\
&=&U_{j,m}^{n-1}+\sigma \left( 1-2\beta \right) \left[
V_{j-1,m}^{n}-2V_{j,m}^{n}+V_{j+1,m}^{n}+V_{j,m-1}^{n}-2V_{j,m}^{n}+V_{j,m+1}^{n}%
\right]  \\
&&+\sigma \beta \left[
V_{j-1,m}^{n-1}-2V_{j,m}^{n-1}+V_{j+1,m}^{n-1}+V_{j,m-1}^{n-1}-2V_{j,m}^{n-1}+V_{j,m+1}^{n-1}%
\right] +\sigma \lambda F_{j,m}^{n}.
\end{eqnarray*}%
Taking into account the boundary conditions, we obtain the full matrix
expression

\begin{eqnarray}
U^{n+1}-\sigma \beta \left( AV^{n+1}+V^{n+1}A^{T}\right) &=&U^{n-1}+\sigma
\left( 1-2\beta \right) \left( AV^{n}+V^{n}A^{T}\right)  \notag \\
&&+\sigma \beta \left( AV^{n-1}+V^{n-1}A^{T}\right) +\sigma \lambda F^{n}
\label{FinalDiscrete1}
\end{eqnarray}%
where
\begin{equation*}
A=\left(
\begin{array}{cccccc}
-2 & 2 & 0 & \cdots & \cdots & 0 \\
1 & -2 & 1 & 0 & \cdots & 0 \\
0 & \ddots & \ddots & \ddots & \ddots & \vdots \\
\vdots & \ddots & \ddots & \ddots & \ddots & 0 \\
0 & \cdots & 0 & 1 & -2 & 1 \\
0 & \cdots & \cdots & 0 & 2 & -2%
\end{array}%
\right)
\end{equation*}%
$U^{n}=\left( U_{j,m}^{n}\right) _{0\leq j,m\leq J}$ and $V=\left(
V_{j,m}^{n}\right) _{0\leq j,m\leq J}$.

Next, for $Q\in M_{\left( J+1\right) ^{2}}\left( \mathbb{R}\right) ,$ we
denote $\mathcal{L}_{Q}$ the linear operator which associates to $X\in
M_{\left( J+1\right) ^{2}}\left( \mathbb{R}\right) $ its image $\mathcal{L}%
_{Q}\left( X\right) =QX+XQ^{T}$. Thus, equation (\ref{FianlDiscrete1}) will
be written as
\begin{equation}
U^{n+1}-\sigma \beta \mathcal{L}_{A}\left( V^{n+1}\right) =U^{n-1}+\sigma
\left( 1-2\beta \right) \mathcal{L}_{A}\left( V^{n}\right) +\sigma \beta
\mathcal{L}_{A}\left( V^{n-1}\right) +\sigma \lambda F^{n}.
\end{equation}%
Now, applying similar techniques as previously we obtain
\begin{equation*}
V_{j,m}^{n,\beta }=qU_{j,m}^{n,\alpha }-\frac{\kappa }{h^{2}}\left(
U_{j-1,m}^{n,\alpha }-2U_{j,m}^{n,\alpha }+U_{j+1,m}^{n,\alpha
}+U_{j,m-1}^{n,\alpha }-2U_{j,m}^{n,\alpha }+U_{j,m+1}^{n,\alpha }\right) .
\end{equation*}%
Let $\delta =\frac{1}{h^{2}}$. This results in the following equation
\begin{eqnarray*}
\beta V_{j,m}^{n+1}+\left( 1-2\beta \right) V_{j,m}^{n}+\beta V_{j,m}^{n-1}
&=&q\left[ \alpha U_{j,m}^{n+1}+\left( 1-2\alpha \right) U_{j,m}^{n}+\alpha
U_{j,m}^{n-1}\right]  \\
&&-\delta \kappa \lbrack \alpha U_{j-1,m}^{n+1}+\left( 1-2\alpha \right)
U_{j-1,m}^{n}+\alpha U_{j-1,m}^{n-1} \\
&&-2\alpha U_{j,m}^{n+1}-2\left( 1-2\alpha \right) U_{j,m}^{n}-2\alpha
U_{j,m}^{n-1} \\
&&+\alpha U_{j+1,m}^{n+1}+\left( 1-2\alpha \right) U_{j+1,m}^{n}+\alpha
U_{j+1,m}^{n-1} \\
&&+\alpha U_{j,m-1}^{n+1}+\left( 1-2\alpha \right) U_{j,m-1}^{n}+\alpha
U_{j,m-1}^{n-1} \\
&&-2\alpha U_{j,m}^{n+1}-2\left( 1-2\alpha \right) U_{j,m}^{n}-2\alpha
U_{j,m}^{n-1} \\
&&+\alpha U_{j,m+1}^{n+1}+\left( 1-2\alpha \right) U_{j,m+1}^{n}+\alpha
U_{j,m+1}^{n-1}],
\end{eqnarray*}%
or equivalently,

\begin{eqnarray*}
&&\beta V_{j,m}^{n+1}-q\alpha U_{j,m}^{n+1} \\
&&+\delta \kappa \alpha \left[
U_{j-1,m}^{n+1}-2U_{j,m}^{n+1}+U_{j+1,m}^{n+1}+U_{j,m-1}^{n+1}-2U_{j,m}^{n+1}+U_{j,m+1}^{n+1}%
\right] \\
&=&-\left( 1-2\beta \right) V_{j,m}^{n}-\beta V_{j,m}^{n-1}+q\left[ \left(
1-2\alpha \right) U_{j,m}^{n}+\alpha U_{j,m}^{n-1}\right] \\
&&-\delta \kappa \lbrack \left( 1-2\alpha \right) \left[
U_{j-1,m}^{n}-2U_{j,m}^{n}+U_{j+1,m}^{n}+U_{j,m-1}^{n}-2U_{j,m}^{n}+U_{j,m+1}^{n}%
\right] \\
&&+\left[
U_{j-1,m}^{n-1}-2U_{j,m}^{n-1}+U_{j+1,m}^{n-1}+U_{j,m-1}^{n-1}-2U_{j,m}^{n-1}+U_{j,m+1}^{n-1}%
\right] ].
\end{eqnarray*}%
Now, applying the boundary conditions, we obtain
\begin{eqnarray}
&&\beta V^{n+1}-\alpha \left[ qU^{n+1}-\delta \kappa \mathcal{L}_{A}\left(
U^{n+1}\right) \right]  \notag \\
&=&\left( 1-2\alpha \right) \left( qU^{n}-\delta \kappa \mathcal{L}%
_{A}\left( U^{n}\right) \right) +\alpha \left( qU^{n-1}-\delta \kappa
\mathcal{L}_{A}\left( U^{n-1}\right) \right)  \notag \\
&&-\left( 1-2\beta \right) V^{n}-\beta V^{n-1}.
\end{eqnarray}%
As a result, we obtain finally the following discrete coupled system.
\begin{equation}
\left\{
\begin{array}{l}
U^{n+1}-\sigma \beta \mathcal{L}_{A}\left( V^{n+1}\right) =U^{n-1}+\sigma
\left( 1-2\beta \right) \mathcal{L}_{A}\left( V^{n}\right) +\sigma \beta
\mathcal{L}_{A}\left( V^{n-1}\right) +\sigma \lambda F^{n} \\
\beta V^{n+1}-\alpha \left[ qU^{n+1}-\delta \kappa \mathcal{L}_{A}\left(
U^{n+1}\right) \right] =\left( 1-2\alpha \right) \left( qU^{n}-\delta \kappa
\mathcal{L}_{A}\left( U^{n}\right) \right) \\
\ \ \ \ \ \ \ \ \ \ \ \ \ \ \ \ \ \ \ \ \ \ \ \ \ \ \ \ \ \ \ \ \ \ \ \ \ \
\ \ \ \ \ \ \ \ \ \ \ \ \ \ \ \ +\alpha \left( qU^{n-1}-\delta \kappa
\mathcal{L}_{A}\left( U^{n-1}\right) \right) \\
\ \ \ \ \ \ \ \ \ \ \ \ \ \ \ \ \ \ \ \ \ \ \ \ \ \ \ \ \ \ \ \ \ \ \ \ \ \
\ \ \ \ \ \ \ \ \ \ \ \ \ \ \ \ \ -\left( 1-2\beta \right) V^{n}-\beta
V^{n-1}%
\end{array}%
\right.  \label{LyapunovSystem}
\end{equation}

\section{Solvability of the discret method}

In this section we will examine the solvability of the discrete scheme. the
main idea is by transforming the system $\left(\ref{LyapunovSystem}\right)$
to an equality of the forme $\left(U^{n+1},V^{n+1}\right)=\phi%
\left(U^{n},V^{n},U^{n-1},V^{n-1}\right)$ with an appropriate function $\phi$%
. We prove precisely that $\phi$ can be expressed with general
Lyapunov-Sylvester form. Next using general properties of such operators we
prove that $\phi $ is an isomorphism. In most studies, even recent ones such
as \cite{Benmabrouk2} the authors proved that $\phi$ or some modified
versions may be contractive by inserting translation-dilation parameters
leading to fixed point theory. In the present context it seams that such
transformation is not possible as $\left\Vert\phi\right\Vert$ may be greater
than $1$ and thus no contraction may occur. To overcome this problem we come
back to differential calculus and topological properties.

\begin{theorem}
The system $\left(\ref{LyapunovSystem}\right)$ is uniquely solvable whenever
$U^{0}$ and $U^{1}$ are known.
\end{theorem}

The proof of this result is based on the following preliminary lemma.

\begin{lemma}
Let $E$ be a finite dimensional vector space on $\mathbb{R}$ (or $\mathbb{C}$%
), and $\left(\phi_{n}\right)_{n}$ be a sequence of endomorphisms converging
uniformly to an invertible endomorphism $\phi$. Then there exist $n_{0}$
such that the endomorphism $\phi_{n}$ is invertible for all $n\geq n_{0}$.
\end{lemma}

\begin{proof}
Consider the endomorphism $\phi $ on $M_{\left( J+1\right) ^{2}}\left(
\mathbb{R}\right) \times M_{\left( J+1\right) ^{2}}\left( \mathbb{R}\right) $
defined by
\begin{equation}
\phi \left( X,Y\right) =\left( X-\sigma \beta \mathcal{L}_{A}\left( Y\right)
,\beta Y-\alpha \Gamma \left( X\right) \right)
\end{equation}%
where $\Gamma \left( X\right) =qX-\delta \kappa \mathcal{L}_{A}\left(
X\right) .$ To prove Theorem 2, we show that $\phi $ is a one to one, and so
$\ker \phi $ is reduced to $0.$ Indeed,%
\begin{equation*}
\phi \left( X,Y\right) =0\Leftrightarrow \left( X-\sigma \beta \mathcal{L}%
_{A}\left( Y\right) ,\beta Y-\alpha \left[ qX-\delta \kappa \mathcal{L}%
_{A}\left( X\right) \right] \right) =\left( 0,0\right) .
\end{equation*}%
Which is equivalent to
\begin{equation}
\left\{
\begin{array}{l}
X=\sigma \beta \mathcal{L}_{A}\left( Y\right) \\
\beta Y=\alpha \Gamma \left( X\right)%
\end{array}%
\right.  \label{SolvabiltySystem}
\end{equation}%
for $\beta \neq 0$ we get%
\begin{equation}
Y=\sigma \alpha \Gamma \mathcal{L}_{A}\left( Y\right)
\end{equation}%
Choosing $l=o\left( h^{s+4}\right) $ (which is always possible), the
operator $\mathcal{K}=I-\sigma \alpha \Gamma \mathcal{L}_{A}$ tends
uniformly to $I$ whenever $h$ tends to zero. Indeed, denote%
\begin{equation*}
W=\sigma \alpha \left( qA-\delta \kappa A^{2}\right) =\frac{2l}{h^{2}}\alpha
\left( qA-\frac{1}{h^{2}}\kappa A^{2}\right) \hbox{.}
\end{equation*}%
We have
\begin{equation*}
\mathcal{K}\left( X\right) =X-\sigma \alpha \Gamma \mathcal{L}_{A}\left(
X\right) =Y-WX-XW^{T}+2\sigma \alpha \delta \kappa AXA^{T}.
\end{equation*}%
\newline
thus,
\begin{eqnarray*}
\left\Vert \left( \mathcal{K}-I\right) X\right\Vert &=&\left\Vert
WX+XW^{T}+2\left( \sigma \alpha \delta \kappa \right) AXA^{T}\right\Vert \\
&\leq &2\left\Vert W\right\Vert \left\Vert X\right\Vert +32\sigma \alpha
\delta \left\vert \kappa \right\vert \left\Vert X\right\Vert
\end{eqnarray*}%
\newline
Since we have $\left\Vert W\right\Vert \leq 4\sigma \left\vert \alpha
\right\vert \left[ \left\vert q\right\vert +4\delta \left\vert \kappa
\right\vert \right] ,$ we obtain,%
\begin{equation*}
\left\Vert \left( \mathcal{K}-I\right) X\right\Vert \leq 8\sigma \alpha %
\left[ \left\vert q\right\vert +8\delta \left\vert \kappa \right\vert \right]
\left\Vert X\right\Vert
\end{equation*}%
For $l=o\left( h^{4+s}\right) $ this implies that,%
\begin{equation*}
\left\Vert \left( \mathcal{K}\left( X\right) -I\left( X\right) \right)
\right\Vert \leq 16\alpha \left[ \left\vert q\right\vert
h^{2+s}+8h^{s}\left\vert \kappa \right\vert \right] \left\Vert X\right\Vert
\end{equation*}%
Consequentely the operator $\mathcal{K}$ converge uniformly to the identity
whenevr $h$ tends towred $0$ and $l=o\left( h^{4+s}\right) ,$ withs $s>0.$
Thus, using Lemma 1 $\phi $ is invertible for $l$,$h$ small enough with $%
l=o\left( h^{4+s}\right) .$\newline
For $\beta =0$ we obtain the system%
\begin{equation}
\left\{
\begin{array}{l}
U^{n+1}=U^{n-1}+\sigma \mathcal{L}_{A}\left( V^{n}\right) +\sigma \lambda
F^{n} \\
V^{n}=\alpha \Gamma \left( U^{n+1}\right) +\left( 1-2\alpha \right) \Gamma
\left( U^{n}\right) +\alpha \Gamma \left( U^{n-1}\right)%
\end{array}%
\right.  \label{ThesystemB=0}
\end{equation}%
and thus,%
\begin{equation}
U^{n+1}-\sigma \alpha \Gamma \mathcal{L}_{A}\left( U^{n+1}\right) =\sigma
\lbrack \left( 1-2\alpha \right) \Gamma \mathcal{L}_{A}\left( U^{n}\right)
+U^{n-1}+\alpha \Gamma \mathcal{L}_{A}\left( U^{n-1}\right) +\lambda F^{n}]
\end{equation}%
For the same assumption on $l$ and $h$ as above the same operator $\mathcal{K%
}\left( X\right) =X-\sigma \alpha \Gamma \mathcal{L}_{A}\left( X\right) $
tends toward the identity as $h$ tends to $0$.
\end{proof}

\section{Consistency}

The consistency of the proposed method will be proved by evaluating the
local truncation error arising from the scheme introduced for the
discretization of the system $\left( \ref{Problem1}\right) $

Applying Taylor Taylor's expansion, and assuming that $u$ and $v$ to be
sufficiently differentiable, we get
\begin{equation}
U_{j,m}^{n+1}=u+l\frac{\partial u}{\partial t}+\frac{l^{2}}{2}\frac{\partial
^{2}u}{\partial t^{2}}+\frac{l^{3}}{6}\frac{\partial ^{3}u}{\partial t^{3}}+%
\frac{l^{4}}{24}\frac{\partial ^{4}u}{\partial t^{4}}+...  \label{Un+1jm}
\end{equation}%
Similarly,
\begin{equation}
U_{j,m}^{n-1}=u-l\frac{\partial u}{\partial t}+\frac{l^{2}}{2}\frac{\partial
^{2}u}{\partial t^{2}}-\frac{l^{3}}{6}\frac{\partial ^{3}u}{\partial t^{3}}+%
\frac{l^{4}}{24}\frac{\partial ^{4}u}{\partial t^{4}}+...  \label{Un-1jm}
\end{equation}%
Hence,
\begin{equation}
\frac{U_{j,m}^{n+1}-U_{j,m}^{n-1}}{2l}=\frac{\partial u}{\partial t}+\frac{%
l^{2}}{6}\frac{\partial ^{3}u}{\partial t^{3}}+...
\end{equation}%
Next, we get also
\begin{eqnarray}
V_{j-1,m}^{n-1} &=&\left[ v-l\frac{\partial v}{\partial t}+\frac{l^{2}}{2}%
\frac{\partial ^{2}v}{\partial t^{2}}-\frac{l^{3}}{6}\frac{\partial ^{3}v}{%
\partial t^{3}}+\frac{l^{4}}{24}\frac{\partial ^{4}v}{\partial t^{4}}\right]
\notag \\
&&-h\left[ \frac{\partial v}{\partial x}-l\frac{\partial ^{2}v}{\partial
t\partial x}+\frac{l^{2}}{2}\frac{\partial ^{3}v}{\partial t^{2}\partial x}-%
\frac{l^{3}}{6}\frac{\partial ^{4}v}{\partial t^{3}\partial x}+\frac{l^{4}}{%
24}\frac{\partial ^{5}v}{\partial t^{4}\partial x}\right]  \notag \\
&&+\frac{h^{2}}{2}\left[ \frac{\partial ^{2}v}{\partial x^{2}}-l\frac{%
\partial ^{3}v}{\partial t\partial x^{2}}+\frac{l^{2}}{2}\frac{\partial ^{4}v%
}{\partial t^{2}\partial x^{2}}-\frac{l^{3}}{6}\frac{\partial ^{5}v}{%
\partial t^{3}\partial x^{2}}+\frac{l^{4}}{24}\frac{\partial ^{6}v}{\partial
t^{4}\partial x^{2}}\right]  \notag \\
&&-\frac{h^{3}}{6}\left[ \frac{\partial ^{3}v}{\partial x^{3}}-l\frac{%
\partial ^{4}v}{\partial t\partial x^{3}}+\frac{l^{2}}{2}\frac{\partial ^{5}v%
}{\partial t^{2}\partial x^{3}}-\frac{l^{3}}{6}\frac{\partial ^{6}v}{%
\partial t^{3}\partial x^{3}}+\frac{l^{4}}{24}\frac{\partial ^{7}v}{\partial
t^{4}\partial x^{3}}\right]  \notag \\
&&+\frac{h^{4}}{24}\left[ \frac{\partial ^{4}v}{\partial x^{4}}-l\frac{%
\partial ^{5}v}{\partial t\partial x^{4}}+\frac{l^{2}}{2}\frac{\partial ^{6}v%
}{\partial t^{2}\partial x^{4}}-\frac{l^{3}}{6}\frac{\partial ^{7}v}{%
\partial t^{3}\partial x^{4}}+\frac{l^{4}}{24}\frac{\partial ^{8}v}{\partial
t^{4}\partial x^{4}}\right] +...
\end{eqnarray}%
and
\begin{equation}
V_{j-1,m}^{n}=v-h\frac{\partial v}{\partial x}+\frac{h^{2}}{2}\frac{\partial
^{2}v}{\partial x^{2}}-\frac{h^{3}}{6}\frac{\partial ^{3}v}{\partial x^{3}}+%
\frac{h^{4}}{24}\frac{\partial ^{4}v}{\partial x^{4}}+...
\end{equation}%
and finally,
\begin{eqnarray}
V_{j-1,m}^{n+1} &=&\left[ v+l\frac{\partial v}{\partial t}+\frac{l^{2}}{2}%
\frac{\partial ^{2}v}{\partial t^{2}}+\frac{l^{3}}{6}\frac{\partial ^{3}v}{%
\partial t^{3}}+\frac{l^{4}}{24}\frac{\partial ^{4}v}{\partial t^{4}}\right]
\notag \\
&&-h\left[ \frac{\partial v}{\partial x}+l\frac{\partial ^{2}v}{\partial
t\partial x}+\frac{l^{2}}{2}\frac{\partial ^{3}v}{\partial t^{2}\partial x}+%
\frac{l^{3}}{6}\frac{\partial ^{4}v}{\partial t^{3}\partial x}+\frac{l^{4}}{%
24}\frac{\partial ^{5}v}{\partial t^{4}\partial x}\right]  \notag \\
&&+\frac{h^{2}}{2}\left[ \frac{\partial ^{2}v}{\partial x^{2}}+l\frac{%
\partial ^{3}v}{\partial t\partial x^{2}}+\frac{l^{2}}{2}\frac{\partial ^{4}v%
}{\partial t^{2}\partial x^{2}}+\frac{l^{3}}{6}\frac{\partial ^{5}v}{%
\partial t^{3}\partial x^{2}}+\frac{l^{4}}{24}\frac{\partial ^{6}v}{\partial
t^{4}\partial x^{2}}\right]  \notag \\
&&-\frac{h^{3}}{6}\left[ \frac{\partial ^{3}v}{\partial x^{3}}+l\frac{%
\partial ^{4}v}{\partial t\partial x^{3}}+\frac{l^{2}}{2}\frac{\partial ^{5}v%
}{\partial t^{2}\partial x^{3}}+\frac{l^{3}}{6}\frac{\partial ^{6}v}{%
\partial t^{3}\partial x^{3}}+\frac{l^{4}}{24}\frac{\partial ^{7}v}{\partial
t^{4}\partial x^{3}}\right]  \notag \\
&&+\frac{h^{4}}{24}\left[ \frac{\partial ^{4}v}{\partial x^{4}}+l\frac{%
\partial ^{5}v}{\partial t\partial x^{4}}+\frac{l^{2}}{2}\frac{\partial ^{6}v%
}{\partial t^{2}\partial x^{4}}+\frac{l^{3}}{6}\frac{\partial ^{7}v}{%
\partial t^{3}\partial x^{4}}+\frac{l^{4}}{24}\frac{\partial ^{8}v}{\partial
t^{4}\partial x^{4}}\right] +...
\end{eqnarray}%
Thus,
\begin{eqnarray}
V_{j-1,m}^{n,\beta } &=&v+\beta l^{2}\frac{\partial ^{2}v}{\partial t^{2}}%
+\beta \frac{l^{4}}{12}\frac{\partial ^{4}v}{\partial t^{4}}  \notag \\
&&+h\left[ -\frac{\partial v}{\partial x}-\beta l^{2}\frac{\partial ^{3}v}{%
\partial t^{2}\partial x}-\beta \frac{l^{4}}{12}\frac{\partial ^{5}v}{%
\partial t^{4}\partial x}\right]  \notag \\
&&+\frac{h^{2}}{2}\left[ \frac{\partial ^{2}v}{\partial x^{2}}+\beta l^{2}%
\frac{\partial ^{4}v}{\partial t^{2}\partial x^{2}}+\beta \frac{l^{4}}{12}%
\frac{\partial ^{6}v}{\partial t^{4}\partial x^{2}}\right]  \notag \\
&&+\frac{h^{3}}{6}\left[ -\frac{\partial ^{3}v}{\partial x^{3}}-\beta l^{2}%
\frac{\partial ^{5}v}{\partial t^{2}\partial x^{3}}-\beta \frac{l^{4}}{12}%
\frac{\partial ^{7}v}{\partial t^{4}\partial x^{3}}\right]  \notag \\
&&+\frac{h^{4}}{24}\left[ \frac{\partial ^{4}v}{\partial x^{4}}+\beta l^{2}%
\frac{\partial ^{6}v}{\partial t^{2}\partial x^{4}}+\beta \frac{l^{4}}{12}%
\frac{\partial ^{8}v}{\partial t^{4}\partial x^{4}}\right] +...
\end{eqnarray}%
Similarly,
\begin{equation}
V_{j,m}^{n,\beta }=v+\beta l^{2}\frac{\partial ^{2}v}{\partial t^{2}}+\beta
\frac{l^{4}}{12}\frac{\partial ^{4}v}{\partial t^{4}}+...
\end{equation}%
We have also
\begin{eqnarray}
V_{j+1,m}^{n,\beta } &=&v+\beta l^{2}\frac{\partial ^{2}v}{\partial t^{2}}%
+\beta \frac{l^{4}}{12}\frac{\partial ^{4}v}{\partial t^{4}}  \notag \\
&&+h\left[ \frac{\partial v}{\partial x}+\beta l^{2}\frac{\partial ^{3}v}{%
\partial t^{2}\partial x}+\beta \frac{l^{4}}{12}\frac{\partial ^{5}v}{%
\partial t^{4}\partial x}\right]  \notag \\
&&+\frac{h^{2}}{2}\left[ \frac{\partial ^{2}v}{\partial x^{2}}+\beta l^{2}%
\frac{\partial ^{4}v}{\partial t^{2}\partial x^{2}}+\beta \frac{l^{4}}{12}%
\frac{\partial ^{6}v}{\partial t^{4}\partial x^{2}}\right]  \notag \\
&&+\frac{h^{3}}{6}\left[ \frac{\partial ^{3}v}{\partial x^{3}}+\beta l^{2}%
\frac{\partial ^{5}v}{\partial t^{2}\partial x^{3}}+\beta \frac{l^{4}}{12}%
\frac{\partial ^{7}v}{\partial t^{4}\partial x^{3}}\right]  \notag \\
&&+\frac{h^{4}}{24}\left[ \frac{\partial ^{4}v}{\partial x^{4}}+\beta l^{2}%
\frac{\partial ^{6}v}{\partial t^{2}\partial x^{4}}+\beta \frac{l^{4}}{12}%
\frac{\partial ^{8}v}{\partial t^{4}\partial x^{4}}\right] +...
\end{eqnarray}%
Finally,
\begin{eqnarray}
\frac{V_{j,m-1}^{n,\beta }-2V_{j,m}^{n,\beta }+V_{j,m+1}^{n,\beta }}{h^{2}}
&=&\left[ \frac{\partial ^{2}v}{\partial y^{2}}+\beta l^{2}\frac{\partial
^{4}v}{\partial t^{2}\partial y^{2}}+\beta \frac{l^{4}}{12}\frac{\partial
^{6}v}{\partial t^{4}\partial y^{2}}\right]  \notag \\
&&+\frac{h^{2}}{12}\left[ \frac{\partial ^{4}v}{\partial y^{4}}+\beta l^{2}%
\frac{\partial ^{6}v}{\partial t^{2}\partial y^{4}}+\beta \frac{l^{4}}{12}%
\frac{\partial ^{8}v}{\partial t^{4}\partial y^{4}}\right] +...
\end{eqnarray}%
Now,%
\begin{eqnarray}
&&\frac{U_{j,m}^{n+1}-U_{j,m}^{n-1}}{2l}  \notag \\
&&-\left[ \frac{V_{j-1,m}^{n,\beta }-2V_{j,m}^{n,\beta }+V_{j+1,m}^{n,\beta }%
}{h^{2}}+\frac{V_{j,m-1}^{n,\beta }-2V_{j,m}^{n,\beta }+V_{j,m+1}^{n,\beta }%
}{h^{2}}\right]  \notag \\
&=&\frac{\partial u}{\partial t}-\Delta v+\frac{l^{2}}{6}\frac{\partial ^{3}u%
}{\partial t^{3}}-\beta l^{2}\frac{\partial ^{2}}{\partial t^{2}}\left(
\Delta v\right) -\frac{h^{2}}{12}\left( \frac{\partial ^{4}v}{\partial x^{4}}%
+\frac{\partial ^{4}v}{\partial y^{4}}\right)  \notag \\
&&+\beta l^{2}\frac{\partial ^{2}}{\partial t^{2}}\left( \frac{\partial ^{4}v%
}{\partial x^{4}}+\frac{\partial ^{4}v}{\partial y^{4}}\right) -\beta \frac{%
l^{4}}{12}\left[ \frac{\partial ^{6}v}{\partial t^{4}\partial x^{2}}+\frac{%
\partial ^{6}v}{\partial t^{4}\partial y^{2}}\right]  \notag \\
&&-\beta \frac{l^{4}}{12}\frac{h^{2}}{12}\left[ \frac{\partial ^{8}v}{%
\partial t^{4}\partial x^{4}}+\frac{\partial ^{8}v}{\partial t^{4}\partial
y^{4}}\right] +...
\end{eqnarray}%
We now examine the second equation in $\left( \ref{Problem1}\right) $.
Applying the same calculus as above, we get
\begin{equation*}
v=\left( qv-\kappa \left( \Delta u\right) \right) -\beta l^{2}\frac{\partial
^{2}v}{\partial t^{2}}+q\alpha l^{2}\frac{\partial ^{2}u}{\partial t^{2}}%
-\kappa \alpha l^{2}\frac{\partial \left( \Delta u\right) }{\partial t^{2}}%
-\kappa \frac{h^{2}}{12}\left( \frac{\partial ^{4}u}{\partial x^{4}}+\frac{%
\partial ^{4}u}{\partial y^{4}}\right) +o\left( l^{2}+h^{2}\right) .
\end{equation*}%
It results from above that the principal part of the first equation in
system \ref{Problem1} is
\begin{eqnarray}
\mathcal{L}_{u,v}^{1}\left( t,x,y\right) &=&\beta l^{2}\frac{\partial ^{2}v}{%
\partial t^{2}}-\alpha ql^{2}\frac{\partial ^{2}u}{\partial t^{2}}-\alpha
\kappa l^{2}\frac{\partial ^{2}\left( \Delta v\right) }{\partial t^{2}}
\notag \\
&&-\kappa \frac{l^{2}}{12}\left( \frac{\partial ^{4}v}{\partial x^{4}}+\frac{%
\partial ^{4}v}{\partial y^{4}}\right) +o\left( l^{2}+h^{2}\right) .
\end{eqnarray}%
The principal part of the local error truncation due to the second equation
is
\begin{eqnarray}
\mathcal{L}_{u,v}^{2}\left( t,x,y\right) &=&\beta \frac{l^{2}}{2}\frac{%
\partial ^{2}}{\partial t^{2}}\left( v-qu-\kappa \alpha \Delta u\right)
\notag \\
&&-\kappa \frac{h^{2}}{12}\left( \frac{\partial ^{4}v}{\partial x^{4}}+\frac{%
\partial ^{4}v}{\partial y^{4}}\right) +o\left( l^{2}+h^{2}\right) .
\end{eqnarray}%
As a result, we get the following lemma.

\begin{lemma}
The numerical method is consistent with an order $2$ in space and time.
\end{lemma}

\begin{proof}
It is clear that the two operators $\mathcal{L}_{u,v}^{1}$ and $\mathcal{L}%
_{u,v}^{2}$ tend towards $0$ as $l$ and $h$ tend to $0$, which ensures the
consistency of the method. Furthermore, the method is consistent with an
order $2$ in time and space.
\end{proof}

\section{Stability and convergence}

The stability of the discrete scheme will be evaluated using Lyapunov
criterion which states that a linear system $\mathcal{L}\left(
x_{n+1},x_{n},x_{n-1},...\right) =0$ is stable in the sense of Lyapunov if
for any bounded initial solution $x_{0},$ the solution $x_{n}$ remains
bounded for all $n\geq 0.$ In this section we prove precisely the following
result.

\begin{lemma}
The solution $\left( U^{n},V^{n}\right) $ is bounded independently of $n$
whenever the initial solution $\left( U^{0},V^{0}\right) $ is.
\end{lemma}

\begin{proof}
We proceed by recurrence on $n.$ Assume $\left\Vert \left(
U_{0},V_{0}\right) \right\Vert \leq \eta $ for some positive $\eta $. The
system $\left( \ref{LyapunovSystem}\right) $ can be written on the form%
\begin{equation}
\left\{
\begin{array}{c}
U^{n+1}-\sigma \beta \mathcal{L}_{A}\left( V^{n+1}\right) =U^{n-1}+\sigma
\left( 1-2\beta \right) \mathcal{L}_{A}\left( V^{n}\right) +\sigma \beta
\mathcal{L}_{A}\left( V^{n-1}\right) +\sigma \lambda F^{n}. \\
\beta V^{n+1}=\alpha \Gamma \left( U^{n+1}\right) +\left( 1-2\alpha \right)
\Gamma \left( U^{n}\right) +\alpha \Gamma \left( U^{n-1}\right) -\left(
1-2\beta \right) V^{n}-\beta V^{n-1}.%
\end{array}%
\right.  \label{StabilitySystem}
\end{equation}%
The last equation gives,%
\begin{eqnarray*}
\beta \mathcal{L}_{A}\left( V^{n+1}\right) &=&\alpha \mathcal{L}_{A}\Gamma
\left( U^{n+1}\right) +\left( 1-2\alpha \right) \mathcal{L}_{A}\Gamma \left(
U^{n}\right) \\
&&+\alpha \mathcal{L}_{A}\Gamma \left( U^{n-1}\right) -\left( 1-2\beta
\right) \mathcal{L}_{A}\left( V^{n}\right) -\beta \mathcal{L}_{A}\left(
V^{n-1}\right) .
\end{eqnarray*}%
Substituting in the first one, we obtain%
\begin{equation}
\mathcal{K}\left( U^{n+1}\right) =\sigma \left( 1-2\alpha \right) \mathcal{L}%
_{A}\Gamma \left( U^{n}\right) +\sigma \alpha \mathcal{L}_{A}\Gamma \left(
U^{n-1}\right) +U^{n-1}+\sigma \lambda F^{n}.  \label{TUn+1}
\end{equation}%
\newline
Next, recall that%
\begin{equation*}
\left\vert F_{j,m}^{n}\right\vert =\frac{1}{4}\left\vert \left(
U_{j+1,m}^{n}-U_{j-1,m}^{n}\right) ^{2}+\left(
U_{j,m+1}^{n}-U_{j,m-1}^{n}\right) ^{2}\right\vert .
\end{equation*}%
Thus,%
\begin{equation*}
\left\vert F_{j,m}^{n}\right\vert \leq \frac{1}{2}\left[ \left(
U_{j+1,m}^{n}\right) ^{2}+\left( U_{j-1,m}^{n}\right) ^{2}+\left(
U_{j,m+1}^{n}\right) ^{2}+\left( U_{j,m-1}^{n}\right) ^{2}\right]
\end{equation*}%
and consequently,
\begin{equation}
\left\Vert F^{n}\right\Vert _{2}\leq 2\left\Vert U^{n}\right\Vert _{2}^{2}.
\label{NormeFn}
\end{equation}%
Finally, $\left( \ref{TUn+1}\right) $ yields that%
\begin{eqnarray*}
\left\Vert \mathcal{K}\left( U^{n+1}\right) \right\Vert &\leq &\sigma
\left\vert 1-2\alpha \right\vert \left\Vert \mathcal{L}_{A}\Gamma \left(
U^{n}\right) \right\Vert +\sigma \left\vert \alpha \right\vert \left\Vert
\mathcal{L}_{A}\Gamma \left( U^{n-1}\right) \right\Vert +\left\Vert
U^{n-1}\right\Vert \\
&&+\sigma \left\vert \lambda \right\vert \left\Vert F^{n}\right\Vert .
\end{eqnarray*}%
Setting $\omega =\left\vert q\right\vert +8\delta \left\vert \kappa
\right\vert $, we obtain
\begin{equation}
\left\Vert \mathcal{K}\left( U^{n+1}\right) \right\Vert \leq 8\omega \sigma
\left\vert 1-2\alpha \right\vert \left\Vert U^{n}\right\Vert +\left[
1+8\omega \sigma \left\vert \alpha \right\vert \right] \left\Vert
U^{n-1}\right\Vert +2\sigma \left\vert \lambda \right\vert \left\Vert
U^{n}\right\Vert ^{2}.  \label{NormeTUn+1}
\end{equation}%
\newline
We now evaluate $\left\Vert V^{n+1}\right\Vert .$ Applying $\Gamma $ for the
first equation in the system $\left( \ref{StabilitySystem}\right) $, we get%
\begin{eqnarray*}
\Gamma \left( U^{n+1}\right) &=&\sigma \beta \Gamma \mathcal{L}_{A}\left(
V^{n+1}\right) +\Gamma \left( U^{n-1}\right) +\sigma \left( 1-2\beta \right)
\Gamma \left( \mathcal{L}_{A}\left( V^{n}\right) \right) \\
&&+\sigma \beta \Gamma \left( \mathcal{L}_{A}\left( V^{n-1}\right) \right)
+\sigma \lambda \Gamma \left( F^{n}\right) .
\end{eqnarray*}%
By replacing in the second equation of $\left( \ref{SolvabiltySystem}\right)
$\ we obtain%
\begin{eqnarray}
\beta \mathcal{K}\left( V^{n+1}\right) &=&\left( 1-2\beta \right) \left[
\sigma \alpha \Gamma \left( \mathcal{L}_{A}\left( V^{n}\right) \right) -V^{n}%
\right]  \notag \\
&&+\beta \left[ \sigma \alpha \Gamma \left( \mathcal{L}_{A}\left(
V^{n-1}\right) \right) -V^{n-1}\right]  \notag \\
&&+2\alpha \Gamma \left( U^{n-1}\right) +\left( 1-2\alpha \right) \Gamma
\left( U^{n}\right) +\sigma \alpha \lambda \Gamma \left( F^{n}\right) ,
\label{ThetaVn+1}
\end{eqnarray}%
We get from $\left( \ref{ThetaVn+1}\right) $ and $\left( \ref{NormeFn}%
\right) $,%
\begin{eqnarray}
\left\Vert \beta \mathcal{K}\left( V^{n+1}\right) \right\Vert &\leq
&\left\vert 1-2\beta \right\vert \left[ 8\sigma \left\vert \alpha
\right\vert \omega +1\right] \left\Vert V^{n}\right\Vert
\label{NormThetaVn+1} \\
&&+\left\vert \beta \right\vert \left[ \left[ 8\sigma \left\vert \alpha
\right\vert \omega +1\right] \left\Vert V^{n-1}\right\Vert \right]
+2\left\vert \alpha \right\vert \omega \left\Vert U^{n-1}\right\Vert  \notag
\\
&&+\left( 1-2\alpha \right) \omega \left\Vert U^{n}\right\Vert +2\sigma
\alpha \lambda \omega \left\Vert U^{n}\right\Vert ^{2}.  \notag
\end{eqnarray}%
\newline
Now comming back to $\left( \ref{Problem1}\right) $ and appalying boundry
conditions, we get%
\begin{equation}
U^{-1}=U^{0}+l\widetilde{\varphi }\hbox{ \ \ and \ \ }V^{-1}=qU^{0}+%
\widetilde{\psi }
\end{equation}%
where,%
\begin{equation*}
\widetilde{\varphi }=-q\Delta \varphi +\kappa \Delta ^{2}\varphi -\lambda
\left\vert \nabla \varphi \right\vert ^{2}
\end{equation*}%
and
\begin{equation*}
\widetilde{\psi }=-\left( lq^{2}+\kappa \right) \Delta \varphi +2l\kappa
q\Delta ^{2}\varphi -l\kappa ^{2}\Delta ^{3}\varphi -\lambda lq\left\vert
\nabla \varphi \right\vert ^{2}+\lambda l\kappa \Delta \left( \left\vert
\nabla \varphi \right\vert ^{2}\right) .
\end{equation*}%
Hence,
\begin{equation}
\left\Vert U^{-1}\right\Vert \leq \left\Vert U^{0}\right\Vert +l\left\Vert
\widetilde{\varphi }\right\Vert \hbox{ \ \ and \ \ }\left\Vert
V^{-1}\right\Vert \leq \left\vert q\right\vert \left\Vert U^{0}\right\Vert
+\left\Vert \widetilde{\psi }\right\Vert .  \label{NromeU-1andV-1}
\end{equation}%
Now, the Lyapunov criterion for stability states exactly that
\begin{equation*}
\forall \varepsilon >0,\exists \eta >0:\ \left\Vert \left(
U^{0},V^{0}\right) \right\Vert \leq \eta \Rightarrow \left\Vert \left(
U^{n},V^{n}\right) \right\Vert \leq \varepsilon ,\ \ \ \ \forall n\geq 0.
\end{equation*}%
\newline
For $n=1,$ and any $\varepsilon $ given such that $\left\Vert \left(
U^{1},V^{1}\right) \right\Vert \leq \varepsilon ,$ we seek an $\eta >0$ for
wich $\left\Vert \left( U^{0},V^{0}\right) \right\Vert <\eta $.\newline
By direct substitution in $\left( \ref{NormeTUn+1}\right) $, for $n=0,$ we
obtain%
\begin{equation*}
\left\Vert \mathcal{K}\left( U^{1}\right) \right\Vert \leq 8\omega \sigma
\left\vert 1-2\alpha \right\vert \left\Vert U^{0}\right\Vert +\left[
1+8\omega \sigma \left\vert \alpha \right\vert \right] \left\Vert
U^{-1}\right\Vert +2\sigma \left\vert \lambda \right\vert \left\Vert
U^{0}\right\Vert ^{2}.
\end{equation*}%
From $\left( \ref{NromeU-1andV-1}\right) ,$ we obtain%
\begin{eqnarray*}
\left\Vert \mathcal{K}\left( U^{1}\right) \right\Vert &\leq &2\sigma
\left\vert \lambda \right\vert \left\Vert U^{0}\right\Vert ^{2}+\left(
1+8\omega \sigma \left[ \left\vert 1-2\alpha \right\vert +\left\vert \alpha
\right\vert \right] \right) \left\Vert U^{0}\right\Vert \\
&&+l\left( 1+8\sigma \left\vert \alpha \right\vert \omega \right) \left\Vert
\widetilde{\varphi }\right\Vert .
\end{eqnarray*}%
Observing that,%
\begin{equation*}
\left\vert 1-2\alpha \right\vert +\left\vert \alpha \right\vert \leq \left(
1+3\left\vert \alpha \right\vert \right) ,
\end{equation*}%
we get%
\begin{equation*}
\left\Vert \mathcal{K}\left( U^{1}\right) \right\Vert \leq 2\sigma
\left\vert \lambda \right\vert \left\Vert U^{0}\right\Vert ^{2}+8\omega
\sigma \left( 1+3\left\vert \alpha \right\vert \right) \left\Vert
U^{0}\right\Vert +l\left( 1+8\sigma \left\vert \alpha \right\vert \omega
\right) \left\Vert \widetilde{\varphi }\right\Vert .
\end{equation*}%
Next choosing $l=o\left( h^{4+s}\right) $ small enough, we obtain%
\begin{equation*}
\left\Vert U^{1}\right\Vert \leq 4\sigma \left\vert \lambda \right\vert
\left\Vert U^{0}\right\Vert ^{2}+16\omega \sigma \left( 1+3\left\vert \alpha
\right\vert \right) \left\Vert U^{0}\right\Vert +2l\left( 1+8\sigma
\left\vert \alpha \right\vert \omega \right) \left\Vert \widetilde{\varphi }%
\right\Vert .
\end{equation*}%
Now, for $\varepsilon >0$, we seek $\eta >0$ such that
\begin{equation}
4\sigma \left\vert \lambda \right\vert \eta ^{2}+16\omega \sigma \left(
1+3\left\vert \alpha \right\vert \right) \eta +2l\left( 1+8\sigma \left\vert
\alpha \right\vert \omega \right) \left\Vert \widetilde{\varphi }\right\Vert
<\varepsilon  \label{equality1}
\end{equation}%
or otherwise,%
\begin{equation*}
4\sigma \left\vert \lambda \right\vert \eta ^{2}+16\omega \sigma \left(
1+3\left\vert \alpha \right\vert \right) \eta +2l\left( 1+8\sigma \left\vert
\alpha \right\vert \omega \right) \left\Vert \widetilde{\varphi }\right\Vert
-\varepsilon <0.
\end{equation*}%
The discriminant of the last inequatlity is
\begin{equation*}
\Delta ^{\prime }=64\left( \omega \sigma \left( 1+3\left\vert \alpha
\right\vert \right) \right) ^{2}-4\sigma \left\vert \lambda \right\vert
(2l\left( 1+8\sigma \left\vert \alpha \right\vert \omega \right) \left\Vert
\widetilde{\varphi }\right\Vert -\varepsilon ).
\end{equation*}%
For the same assumption on $l$ and $h$ as above,
\begin{equation*}
\Delta ^{\prime }\sim 64\left( \omega \sigma \left( 1+3\left\vert \alpha
\right\vert \right) \right) ^{2}+4\sigma \left\vert \lambda \right\vert
\varepsilon >0.
\end{equation*}%
Consenquentely, there are two zeros $\eta _{1}<\eta _{2}$ of the inequality
above. Furthermore, replacing $\eta $ with $0$ we get a negative quantity,
thus $\eta _{1}<0<\eta _{2}.$ As a result, $\eta _{2}$ is the good
candidate. Now, choosing $\left\Vert (U^{0},V^{0})\right\Vert \leq \eta _{2}$%
, we get immediately $\left\Vert U^{1}\right\Vert <\varepsilon .$\newline
Next, already with $n=0,$ we get similarly to the previous case%
\begin{eqnarray*}
\left\Vert \beta \mathcal{K}\left( V^{1}\right) \right\Vert &\leq
&\left\vert 1-2\beta \right\vert \left[ 8\sigma \left\vert \alpha
\right\vert \omega +1\right] \left\Vert V^{0}\right\Vert +\left\vert \beta
\right\vert \left[ \left[ 8\sigma \left\vert \alpha \right\vert \omega +1%
\right] \left\Vert V^{-1}\right\Vert \right] \\
&&+2\left\vert \alpha \right\vert \omega \left\Vert U^{-1}\right\Vert
+\left\vert 1-2\alpha \right\vert \omega \left\Vert U^{0}\right\Vert
+2\sigma \alpha \lambda \omega \left\Vert U^{0}\right\Vert ^{2}
\end{eqnarray*}%
\newline
Choosing $l=o\left( h^{4+s}\right) $ small enough as above, and $\mu
=8\sigma \left\vert \alpha \right\vert \omega +1$, we obtain%
\begin{eqnarray*}
\frac{\left\vert \beta \right\vert }{2}\left\Vert V^{1}\right\Vert &\leq
&\left\Vert \beta \mathcal{K}\left( V^{1}\right) \right\Vert \leq \mu
\left\vert 1-2\beta \right\vert \left\Vert V^{0}\right\Vert +\mu \left\vert
\beta \right\vert \left[ \left\Vert V^{-1}\right\Vert \right] \\
&&+2\left\vert \alpha \right\vert \omega \left\Vert U^{-1}\right\Vert
+\left\vert 1-2\alpha \right\vert \omega \left\Vert U^{0}\right\Vert
+2\sigma \alpha \lambda \omega \left\Vert U^{0}\right\Vert ^{2}
\end{eqnarray*}%
Next, recall that
\begin{equation*}
\left\Vert U^{-1}\right\Vert \leq \left\Vert U^{0}\right\Vert +\left\Vert
\widetilde{\varphi }\right\Vert \hbox{ and }\left\Vert V^{-1}\right\Vert
\leq \left\vert q\right\vert \left\Vert U^{0}\right\Vert +l\left\Vert
\widetilde{\psi }\right\Vert ,
\end{equation*}%
we get%
\begin{eqnarray*}
\left\vert \beta \right\vert \left\Vert V^{1}\right\Vert &\leq &2\mu
\left\vert 1-2\beta \right\vert \left\Vert V^{0}\right\Vert +2\mu \left\vert
\beta \right\vert \left( \left\vert q\right\vert \left\Vert U^{0}\right\Vert
+\left\Vert \widetilde{\psi }\right\Vert \right) \\
&&+4\left\vert \alpha \right\vert \omega \left( \left\Vert U^{0}\right\Vert
+l\left\Vert \widetilde{\varphi }\right\Vert \right) +2\left\vert 1-2\alpha
\right\vert \omega \left\Vert U^{0}\right\Vert +4\sigma \alpha \lambda
\omega \left\Vert U^{0}\right\Vert ^{2}.
\end{eqnarray*}%
Henceforth,%
\begin{eqnarray*}
\left\vert \beta \right\vert \left\Vert V^{1}\right\Vert &\leq &4\sigma
\alpha \lambda \omega \left\Vert U^{0}\right\Vert ^{2}+2\mu \left\vert
1-2\beta \right\vert \left\Vert V^{0}\right\Vert \\
&&+2\left( \mu \left\vert q\right\vert \left\vert \beta \right\vert +\omega
\left( 2\left\vert \alpha \right\vert +\left\vert 1-2\alpha \right\vert
\right) \right) \left\Vert U^{0}\right\Vert \\
&&+2\mu \left\vert \beta \right\vert \left\Vert \widetilde{\psi }\right\Vert
+4\left\vert \alpha \right\vert \omega l\left\Vert \widetilde{\varphi }%
\right\Vert .
\end{eqnarray*}%
Now, proceeding as for $U^{1}$, we prove that for all $\varepsilon >0,$
there is an $\eta _{2}^{\prime }>0$ satisfying $\left\Vert V^{1}\right\Vert
<\varepsilon $ whenever $\left\Vert (U^{0},V^{0})\right\Vert \leq \eta
_{2}^{\prime }.$ Finally $\eta =\min \left( \eta _{2},\eta _{2}^{\prime
}\right) $ answers the question.\newline
Assume now that $\left( U^{k},V^{k}\right) $ is bounded for $k=1,2,...,n$ by
$\varepsilon _{1}$ whenever $\left( U^{0},V^{0}\right) $ is bounded by $\eta
$ and let $\varepsilon >0.$ We shall prove that it is possible to choose $%
\eta $ satisfying $\left\Vert \left( U^{n+1},V^{n+1}\right) \right\Vert \leq
\varepsilon .$
\end{proof}

\begin{lemma}
(Banach) Let $E,F$ be two Banach spaces and $\phi :E\rightarrow F$ be
linear. If $\phi $ is continue and bijective then $\phi $ is a homeomorphism.%
\newline
Consider as above the endomorphism $\phi $ on $M_{\left( J+1\right)
^{2}}\left( \mathbb{R}\right) \times M_{\left( J+1\right) ^{2}}\left(
\mathbb{R}\right) $ defined by%
\begin{equation*}
\phi \left( X,Y\right) =\left( X-\sigma \beta \mathcal{L}_{A}\left( Y\right)
,\beta Y-\alpha \left[ qX-\delta \kappa \mathcal{L}_{A}\left( X\right) %
\right] \right) .
\end{equation*}%
Consider also%
\begin{equation*}
f_{1}\left( X,Y,Z,W\right) =X+\alpha \beta \mathcal{L}_{A}\left( Y\right)
+\sigma \left( 1-2\beta \right) +\sigma \beta \mathcal{L}_{A}\left( Z\right)
+\sigma \lambda W
\end{equation*}%
and%
\begin{equation*}
f_{2}\left( X,Y,Z,T\right) =\alpha \Gamma \left( X\right) -\beta Y+\left(
1-2\alpha \right) \Gamma \left( Z\right) -\left( 1-2\beta \right) T.
\end{equation*}%
We obtain thus
\begin{eqnarray*}
\phi \left( U^{n+1},V^{n+1}\right) &=&\left( U^{n+1}-\sigma \beta \mathcal{L}%
_{A}\left( V^{n+1}\right) ,\beta V^{n+1}-\alpha \Gamma \left( U^{n+1}\right)
\right) \\
&=&\left( f_{1}\left( U^{n-1},V^{n-1},V^{n},F^{n}\right) ,f_{2}\left(
U^{n-1},V^{n-1},U^{n},V^{n}\right) \right) .
\end{eqnarray*}%
We have already proved that $\phi $ is one to one for $l$ and $h$ small
enough. Since $\phi $ is linear and continuous, then $\phi $ has a
continuous inverse function. So, $\phi $ is a homeomorphism on $M_{\left(
J+1\right) ^{2}}\left( \mathbb{R}\right) \times M_{\left( J+1\right)
^{2}}\left( \mathbb{R}\right) $ . Furthermore%
\begin{equation*}
\left( U^{n+1},V^{n+1}\right) =\phi ^{-1}\left( f_{1}\left(
U^{n-1},V^{n-1},V^{n},F^{n}\right) ,f_{2}\left(
U^{n-1},V^{n-1},U^{n},V^{n}\right) \right) .
\end{equation*}%
As $\phi ^{-1}$ is continuous, there exists $C>0$ such that
\end{lemma}

\begin{eqnarray*}
\left\Vert \left( U^{n+1},V^{n+1}\right) \right\Vert &=&\left\Vert \phi
^{-1}\left( f_{1}\left( U^{n-1},V^{n-1},V^{n},F^{n}\right) ,f_{2}\left(
U^{n-1},V^{n-1},U^{n},V^{n}\right) \right) \right\Vert \\
&\leq &C\left\Vert \left( f_{1}\left( U^{n-1},V^{n-1},V^{n},F^{n}\right)
,f_{2}\left( U^{n-1},V^{n-1},U^{n},V^{n}\right) \right) \right\Vert .
\end{eqnarray*}%
\newline
By choosing $\left\Vert \left( U^{k},V^{k}\right) \right\Vert \leq \eta $,
for all $k=0,1,...n$, we get%
\begin{equation*}
\left\Vert f_{1}\left( U^{n-1},V^{n-1},V^{n},F^{n}\right) \right\Vert \leq
2\sigma \left\vert \lambda \right\vert \eta ^{2}+\left( 1+8\sigma \left(
1+3\left\vert \beta \right\vert \right) \right) \eta .
\end{equation*}%
Now, it suffices to prove that there exists $\eta >0$ for which%
\begin{equation*}
\left( 1+8\sigma \left( 1+3\left\vert \beta \right\vert \right) \right) \eta
+2\sigma \left\vert \lambda \right\vert \eta ^{2}\leq \varepsilon \ \ \
\Leftrightarrow \ \ \ 2\sigma \left\vert \lambda \right\vert \eta
^{2}+\left( 1+8\sigma \left( 1+3\left\vert \beta \right\vert \right) \right)
\eta -\varepsilon \leq 0.
\end{equation*}%
The discernments is%
\begin{equation*}
\Delta =\left( 1+8\sigma \left( 1+3\left\vert \beta \right\vert \right)
\right) ^{2}+8\sigma \left\vert \lambda \right\vert \varepsilon >0.
\end{equation*}%
Hence, there are two zeros of the last equality
\begin{equation*}
\eta _{1}=\frac{-\left( 1+8\sigma \left( 1+3\left\vert \beta \right\vert
\right) \right) -\sqrt{\Delta }}{4\sigma \left\vert \lambda \right\vert }%
\hbox{ and }\eta _{2}=\frac{-\left( 1+8\sigma \left( 1+3\left\vert \beta
\right\vert \right) \right) +\sqrt{\Delta }}{4\sigma \left\vert \lambda
\right\vert }.
\end{equation*}%
It is straightforward that $0\in ]\eta _{1},\eta _{2}[.$ Hence $\eta _{2}>0$%
. Now for $f_{2}$ we get%
\begin{equation*}
\left\Vert f_{2}\left( U^{n-1},V^{n-1},U^{n},V^{n}\right) \right\Vert \leq
\left[ \omega \left( 1+3\left\vert \alpha \right\vert \right) +\left(
1+3\left\vert \beta \right\vert \right) \right] \eta .
\end{equation*}%
For%
\begin{equation*}
\eta \leq \eta _{3}=\frac{\varepsilon }{\omega \left( 1+3\left\vert \alpha
\right\vert \right) +\left( 1+3\left\vert \beta \right\vert \right) },
\end{equation*}%
we obtain $\left\Vert V^{n+1}\right\Vert \leq \varepsilon .$ Finally,
choosing $\eta $ the minimum between $\eta _{2}$ and $\eta _{3}$ the
criterion is proved.

\begin{lemma}
As the numerical scheme is consistent and stable, it is then convergent.
\end{lemma}

\section{Numerical Implementations}

We propose in this section to present some numerical examples to validate
the theoretical results developed previously. The error between the exact
solutions and the numerical ones via an $L_{2}$ discrete norm will be
estimated. The matrix norm used will be
\begin{equation*}
\Vert X\Vert _{2}=\left( \displaystyle\sum_{i,j=1}^{N}|X_{ij}|^{2}\right)
^{1/2}
\end{equation*}%
for a matrix $X=(X_{ij})\in \mathcal{M}_{N+2}\mathbb{C}$. Denote $u^{n}$ the
net function $u(x,y,t^{n})$ and $U^{n}$ the numerical solution. We propose
to compute the discrete error
\begin{equation}
Er=\displaystyle\max_{n}\Vert U^{n}-u^{n}\Vert _{2}  \label{Er}
\end{equation}%
on the grid $(x_{i},y_{j})$, $0\leq i,j\leq J+1$ and to validate the
convergence rate of the proposed schemes we propose to compute the
proportion
\begin{equation*}
C=\displaystyle\frac{Er}{l^{2}+h^{2}}.
\end{equation*}%
We consider the inhomogeneous problem%
\begin{equation}
\left\{
\begin{array}{l}
\frac{\partial u}{\partial t}=\Delta v+\lambda \left\vert \nabla
u\right\vert ^{2}+g(x,y,t),\quad (x,y,t)\in \Omega \times (t_{0},+\infty )
\\
v=qu-\kappa \Delta u,\quad (x,y,t)\in \Omega \times (t_{0},+\infty ) \\
\left( u,v\right) (x,y,t_{0})=\left( \varphi ,\psi \right) (x,y),\quad
(x,y)\in \overline{\Omega } \\
\overrightarrow{\nabla }(u,v)(x,y,t)=0,\quad (x,y,t)\in \partial \Omega
\times (t_{0},+\infty )%
\end{array}%
\right.  \label{Problem2}
\end{equation}%
on the rectangular domain $\Omega =[-1,1]\times \lbrack -1,1]$, where
\begin{eqnarray}
g(x,y,t) &=&Ke(t)[e(t)[C^{4}(x)S^{2}(x)C^{6}(y)+C^{6}(x)C^{4}(y)S^{2}(y)] \\
&&-\left[ C(x)S^{2}(x)C(y)S^{2}(y)\right] ]
\end{eqnarray}%
and the exact solution
\begin{equation*}
(u,v)(x,y,t)=(e(t)C^{3}(x)C^{3}(y),...),
\end{equation*}%
with

\begin{equation*}
C(x)=\cos (\frac{\pi x}{2})\text{ \ , \ }S(x)=\sin (\frac{\pi x}{2})\text{ \
, \ }e(t)=e^{-9\pi ^{4}t/2}\text{ \ \ and \ \ }K=\frac{9\pi ^{4}}{2}.
\end{equation*}%
In the following tables, numerical results are provided. We computed for
different space and time steps the discrete $L_{2}$-error estimates defined
by (\ref{Er}). The time interval is $[0,1]$ for a choice $t_{0}=0$ and $T=1$%
. The following results are obtained for different values of the parameters $%
J$ (and thus $h$), $l$ ((and thus $N$). The parameters $\alpha $ and $\beta $
are fixed to $\alpha =\frac{1}{3}$ and $\beta =\frac{1}{5}.$ We just notice
that some variations done on these latter parameters have induced an
important variation in the error estimates which explains their effect as
they calibrates the position of the approximated solution around the exact
one. The parameters $q,$ $\lambda $ and $\kappa $ have the role of
viscosity-type coefficients and fixed to the values $\kappa =1,$ $\lambda
=-2\pi ^{2}$ and $q=\frac{11\pi ^{2}}{2}.$ The following tables outputs the
error estimates relatively to the discrete $L^{2}-$norm defined above for
different values of the space and time steps. First, we provide estimates
when the optimal condition $l=o(h^{4+s}),$ $s>0$ is fulfilled. $s$ is fixed
to $0.01$ for the first table. Next, to control more the effect of such
assumption which is due to the presence of a second order Laplacian in the
original problem, we tested the convergence of the scheme at some orders
less than the optimal power fixed to $4.$ The second table provides the
estimates with a slightly sub-critical power $4-s,$ $s>0$ small enough. Here
also $s$ is fixed to $0.01$, and finally in the third table, we tested the
discrete scheme for a strong sub-critical power.

\ \ \ \

\begin{table}[tbp] \centering%
\begin{tabular}{|c|c|c|}
\hline
$J$ & $N$ & $Er2$ \\ \hline
$10$ & $640$ & $1,25.10^{-5}$ \\ \hline
$12$ & $1320$ & $2,46.10^{-6}$ \\ \hline
$14$ & $2450$ & $6,14.10^{-7}$ \\ \hline
$16$ & $4183$ & $1,84.10^{-7}$ \\ \hline
$18$ & $6707$ & $6,38.10^{-8}$ \\ \hline
$20$ & $10233$ & $2,46.10^{-8}$ \\ \hline
$22$ & $14997$ & $1,04.10^{-8}$ \\ \hline
$24$ & $21258$ & $4,76.10^{-9}$ \\ \hline
$25$ & $25039$ & $3,29.10^{-9}$ \\ \hline
$30$ & $52015$ & $6,36.10^{-10}$ \\ \hline
\end{tabular}%
\caption{Error estimates for $l=o(h^{4.01})$}\label{tab1}%
\end{table}%

\begin{table}[tbp] \centering%
\begin{tabular}{|c|c|c|}
\hline
$J$ & $N$ & $Er2$ \\ \hline
$10$ & $616$ & $1,35.10^{-5}$ \\ \hline
$12$ & $1273$ & $2,55.10^{-6}$ \\ \hline
$14$ & $2355$ & $6,65.10^{-7}$ \\ \hline
$16$ & $4012$ & $2,00.10^{-7}$ \\ \hline
$18$ & $6419$ & $6,96.10^{-8}$ \\ \hline
$20$ & $7973$ & $4,06.10^{-8}$ \\ \hline
$22$ & $14295$ & $1,14.10^{-8}$ \\ \hline
$24$ & $20228$ & $5,26.10^{-9}$ \\ \hline
$25$ & $23806$ & $3,64.10^{-9}$ \\ \hline
$30$ & $49273$ & $7,09.10^{-10}$ \\ \hline
\end{tabular}%
\caption{Error estimates for $l=o(h^{3.99})$}\label{tab2}%
\end{table}%

\bigskip

\begin{table}[tbp] \centering%
\begin{tabular}{|c|c|c|}
\hline
$J$ & $N$ & $Er2$ \\ \hline
$10$ & $128$ & $3,10.10^{-4}$ \\ \hline
$12$ & $220$ & $8,82.10^{-5}$ \\ \hline
$14$ & $350$ & $2,99.10^{-5}$ \\ \hline
$16$ & $523$ & $1,17.10^{-5}$ \\ \hline
$18$ & $746$ & $9,59.10^{-6}$ \\ \hline
$20$ & $1024$ & $2,45.10^{-6}$ \\ \hline
$22$ & $1364$ & $1,26.10^{-6}$ \\ \hline
$24$ & $1772$ & $6,84.10^{-7}$ \\ \hline
$25$ & $2004$ & $5,14.10^{-7}$ \\ \hline
$30$ & $3468$ & $1,34.10^{-7}$ \\ \hline
$50$ & $16137$ & $3,96.10^{-9}$ \\ \hline
\end{tabular}%
\caption{Error estimates for $l=o(h^{3.01})$}\label{tab3}%
\end{table}%

\ \

\section{Conclusion}

This paper investigated the solution of the well-known Kuramoto-Sivashinsky
equation in two-dimensional case by applying a two-dimensional finite
difference discretization. The original equation is a 4-th order partial
differential equation. Thus, in a first step it was recasted into a system
of second order partial differential equations by applying a reduction
order. Next, the continuous system of simultaneous coupled PDEs has been
transformed into an algebraic discrete system involving a generalized
Lyapunov-Syslvester type operators. Solvability, consistency, stability and
convergence are then established by applying well-known methods such as
Lax-Richtmyer equivalence theorem and Lyapunov Stability and by examining
the topological properties of the obtained Lyapunov-Sylvester type
operators. The method was finally improved by developing a numerical example.


\begin{thebibliography}{99}
\bibitem{Benachour} S. Benachour,I. Kukavica, W.Rusin, M.Ziane. Anisotropic
Estimates for the Two-Dimensional Kuramoto--Sivashinsky Equation. Springer
Science+Business Media New York 2014. J Dyn Diff Equat DOI
10.1007/s10884-014-9372-3.

\bibitem{Benmabrouk2} A. Ben Mabrouk and M. Ayadi, A linearized
finite-difference method for the solution of some mixed concave and convex
non-linear problems, Appl. Math. Comput. 197 (2008), 1--10.

\bibitem{Benmabrouk1} A. Ben Mabrouk and M. Ayadi, Lyapunov type operators
for numerical solutions of PDEs, Appl. Math. Comput. 204 (2008), 395--407.

\bibitem{Benmabrouk3} A. Ben Mabrouk, M.L. Ben Mohamed, K. Omrani,
Finite-difference approximate solutions for a mixed sub-superlinear
equation, Appl. Math. Comput. 187 (2007) 1007-1016.

\bibitem{Bellout} H. Bellout, S. Benachour, E.Titi, Finite time regularity
versus global regularity for hyperviscous Hamilton--Jacobi-like equations.
Nonlinearity 16, 1967--1989 (2003).

\bibitem{Bezia} A. Bezia, A. Ben Mabrouk and K Betina, Lyapunov-Sylvester
Computational Method for Two-Dimensional Boussinesq Equation. Numerical
Linear Algebra with Applications. resubmitted in revised form, 2015.

\bibitem{Collet} P. Collet, J. Eckmann, H. Epstein, J. Stubbe, A global
attracting set for the Kuramoto--Sivashinsky equation. Comm. Math. Phys.
152(1), 203--214 (1993).

\bibitem{Giacomelli} L. Giacomelli, F. Otto, New bounds for the
Kuramoto--Sivashinsky equation. Comm. Pure Appl. Math. 43, 297--318 (2005).

\bibitem{Goncalves} E. Gon\c{c}alv\`{e}s, Resolution numerique,
discretisation des EDP et EDO. Cours, Institut National Polytechnique de
Grenoble, (2005).

\bibitem{Goodman} J. Goodman, Stability of the Kuramoto--Sivashinsky and
related systems. Comm. Pure Appl. Math. 47(3), 293--306 (1994).

\bibitem{Hansen} J. L. Hansen and T.Bohr, Fractal tracer distributions in
turbulent field theories, arXiv:chao- dyn/9709008V1.

\bibitem{HenriCartan} H. Cartan, Differential Calculus, Kershaw Publishing
Company, London 1971, Translated from the original French text Calcul
differentiel, first published by Hermann in (1967).

\bibitem{Hong} Q.Hong-Ji, J.Yong-Hao, C.Chuan-Fu, H.Li-Hua,Y.Kui,S.Jian-Da,
Dynamic Scaling Behaviour in (2+1)-Dimensional Kuramoto Sivashinsky Model,
CHIN.PHYS.LETT. Vol. 20,5 (2003) 622-625.

\bibitem{Ilyashenko} Y.S. Ilyashenko, : Global analysis of the phase
portrait for the Kuramoto--Sivashinsky equation. J. Dyn. Differ. Equ. 4(4),
585--615 (1992).

\bibitem{Jameson} A. Jameson, Solution of equation $AX+XB=C$ by inversion of
an $M\times M$ or $N\times N$ matrix. SIAM J. Appl Math. 16(5), 1020-1023.
(1968).

\bibitem{Jayaprakash} C. Jayaprakash, F. Hayot, R. Pandit, Phys. Rev. Lett.
71, 12 (1993).

\bibitem{Jia-Li} J. Jia, S. Li, On the inverse and determinant of general
bordered tridiagonal matrices. Computers \& Mathematics with Applications,
69(6), 503-509. (2015).

\bibitem{Kuramoto} Y. Kuramoto, T. Tsuzuki, Persistent Propagation of
Concentration Waves In Dissipative Media Far from Thermal Equilibrium.55,
356 (1976).

\bibitem{Lancaster} P. Lancaster, Explixit solutions of linear matrix
equations.Siam review Vol. 12, No. 4, October (1970).

\bibitem{Nadjafikhah} M. Nadjafikhah, F. Ahangari. Classical and
Nonclassical symmetries of the (2+1)- dimensional Kuramoto-Sivashinsky
equation. School of Mathematics, Iran University of Science and Technology,
Narmak, Tehran 1684613114, Iran.

\bibitem{Nicolaenko} Nicolaenko, B., Scheurer, B., Temam, R.: Some global
dynamical properties of the Kuramoto--Sivashinsky equations: nonlinear
stability and attractors. Phys. D 16(2), 155--183 (1985).

\bibitem{Otto} F. Otto, Optimal bounds on the Kuramoto--Sivashinsky
equation. J. Funct. Anal. 257(7), 2188--2245 (2009).

\bibitem{Procaccia} I. Procaccia, M. H. Jensen, V. S. L'vov, K. Sneppen, R.
Zeitak, Phys. Rev. A 46, 3220 (1992)

\bibitem{Raugel} G. Raugel, G.R. Sell, Navier--Stokes equations on thin 3D
domains. II. Global regularity of spatially periodic solutions. Nonlinear
Partial Differential Equations and Their Applications. Coll\`{e}ge de France
Seminar, vol. XI, pp. 205--247. Longman, Harlow (1994).

\bibitem{Rost} M. Rost, J. Krug, Anisotropic Kuramoto--Sivashinsky equation
for surface growth erosion. Phys. Rev. Lett. bf 75(21), 3894--3897 (1995).

\bibitem{Rionero} S. Rionero, Stability results for hyperbolic and parabolic
equations. Transport Theory \& Statistical Physics, 25(3-5), 323-337. (1996).

\bibitem{Simoncini} V. Simoncini, Computatioanl methods for linear matrix
equations, Course in Dipartimento di Matematica, Universita di Bologna,
Piazza di Porta San Donato 5, I-40127 Bologna, Italia, March 12, (2013).

\bibitem{Sivashinsky} G. I. Sivashinsky, Acta Astronautica 6, 569 (1979).

\bibitem{Sivashinsky2} G. I. Sivashinsky, On the flame propagation under
conditions of stochiometry. SIAM J. Appl. Math. 75,67--82 (1980).

\bibitem{Sell} G.R. Sell, M. Taboada, Local dissipativity and attractors for
the Kuramoto--Sivashinsky equation in thin 2D domains. Nonlin. Anal. 18,
671--687 (1992).

\bibitem{Roth} W. E. Roth, The equations $AX-YB=C$ and $AX-XB=C$ in
matrices, pp. 392-396. Ibid., 3 (1952).

\bibitem{Temam} R. Temam, Infinite-Dimensional Dynamical Systems in
Mechanics and Physics. Springer, Berlin (1988).
\end{thebibliography}
\end{document}